\input louC.sty

\documentclass[a4paper,draft]{amsproc}
\usepackage{amsrefs}
\usepackage{amsmath}
\usepackage{mathrsfs}
\usepackage{amssymb}
\usepackage{amscd} 

\theoremstyle{plain}
 \newtheorem{thm}{\textbf{Theorem}}[section]
 
 \newtheorem{lem}{\textbf{Lemma}}[section]
 
\theoremstyle{definition}

\theoremstyle{remark}
 \newtheorem{rem}{\textbf{Remark}}[section]
 \numberwithin{equation}{section}

\title{Normal Truncated Toeplitz Operators}

\subjclass[2010]{Primary 47}

\author[Chu]{\bfseries Cheng Chu}

\address{
Department of Mathematics \\ 
Vanderbilt University  \\ 
Nashville, Tennessee \\
USA}
\email{cheng.chu@vanderbilt.edu}

\begin{document}

\vspace{18mm}
\setcounter{page}{1}
\thispagestyle{empty}

\begin{abstract}
The characterization of normal truncated Toepltiz operators is first given by Chalendar and Timotin. We give an elementary proof of their result without using the algebraic properties of truncated Toeplitz operators.
\end{abstract}

\maketitle

\section{Introduction}
Let $\DD$ be the open unit disk in the complex plane. Let $L^2$ denote the Lebesgue space of square integrable functions on the unit circle $\partial\DD$. The Hardy space $H^2$ is the subspace of analytic functions on $\DD$ whose Taylor coefficients are square summable. Then it can also be identified with the subspace of $L^2$ of functions whose negative Fourier coefficients vanish. Let $P$ and $P^{\perp}$ be the orthogonal projections from $L^2$ to $H^2$ and $[H^2]^\perp$, respectively. Here $[H^2]^{\perp}$ is the orthogonal complement of $H^2$ in $L^2$.  For $f\in L^\infty$, the space of essentially bounded Lebesgue measurable functions on $\partial\DD$, the Toeplitz operator $T_f$ with symbol $f\in L^\infty$ is defined by $$T_fh=P(fh),$$  for $h\in H^2$. 

An analytic function $\Gt$ is called an inner function if $|\Gt|=1$ a.e. on $\TT$. For each non-constant inner function $\Gt$, the so-called model space is $$K_\Gt=H^2\ominus \Gt H^2.$$ It is a reproducing kernel Hilbert space with reproducing kernels $$k_w^{\Gt}(z)=\frac{1-\overline{\Gt(w)}\Gt(z)}{1-\bar{w}z}.$$

Let $P_\Gt$ denote the orthogonal projection from $L^2$ onto $K_\Gt$,
\beq\label{Pt}
P_{\Gt}f=P f-\Gt P(\bar{\Gt}f).
\eeq
For $\varphi\in L^2$, the truncated Toeplitz operator $A_\phi$ is defined by $$A^{\Gt}_\varphi f=P_\Gt (\varphi f),$$ on the dense subset $K_\Gt \cap H^{\infty}$ of $K_\Gt$. In particular, $K_\Gt \cap H^{\infty}$ contains all reproducing kernels $k_w^{\Gt}$. The operator $A^{\Gt}_\varphi$ may be extended to a bounded operator on $K_\Gt$ even for unbounded symbols $\Gvp$. The symbol $\Gvp$ is never unique and it is proved in \cite{sar07} that $$A^{\Gt}_\varphi=0$$ if and only if 
$$
\Gvp\in \Gt H^2+\ol{\Gt H^2}.
$$
If $\Gt(0)=0$, then $A^{\Gt}_\varphi$ has a unique symbol $$\Gvp\in K_\Gt+\ol{K_\Gt}.$$ The set of all bounded truncated Toeplitz operators is denoted by $\cT_\Gt$.

Recall that a bounded operator $T$ on a Hilbert space $\cH$ is normal if $T^*T=TT^*.$ The characterization of normal truncated Toepltiz operators is first given by Chalendar and Timotin using the algebraic properties of truncated Toeplitz operators obtained by Sarason \cite{sar07} and Sedlock \cite{sed11}. 
\begin{thm}\label{ct}\cite{ct14}*{Theorem 6.2}
Let $\Gt$ be a non-constant inner function vanishing at $0$. Then $A^\Gt_\Gvp$ is normal if and only if one of the following holds
\begin{enumerate}
\item $A^\Gt_\Gvp$ belongs to $\mcB_\Gt^\Ga$, for some unimodular constant $\Ga$.
\item $A^\Gt_\Gvp$ is a linear combination of a self-adjoint truncated Toeplitz operator and the identity. 
\end{enumerate}
\end{thm}
Here $\mcB_\Gt^\Ga$ is a class of truncated Toeplitz operators introduced in \cite{sed11}.
In this note, we give an elementary proof of their result.

\section{Proof of the Main Result}
In this section we offer a proof of our characterization of normal truncated Toepltiz operators $A^{\Gt}_{\Gvp}$. We begin with some reduction. 
Notice that for any constant $C$, $A^{\Gt}_{\Gvp+C}=A^{\Gt}_{\Gvp}+CI, $
which implies $A^{\Gt}_{\Gvp}$ is normal if and only if $A^{\Gt}_{\Gvp+C}$ is normal. Thus we may assume, without losing of generality, that $\Gvp(0)=0$.

For $a\in\DD$, let $u_a$ be the M{\"o}bius transform
$$
u_a(z)=\frac{z-a}{1-\ba z}.
$$
The Crofoot transform is the unitary operator $J: K_\Gt\to K_{ u_a\circ\Gt}$ defined by
$$
J(f)=\frac{\sqrt{1-|a|^2}}{1-\ba \Gt}f.
$$
It is proved in \cite{sar07} that 
$$
J\cT_\Gt J^*=\cT_{u_a\circ\Gt}.
$$
Taking $a=\Gt(0)$, we see that it is sufficient to consider the normal truncated Toeplitz operators for $\Gt(0)=0$.
In this case, constant functions are in $K_{\Gt}$. Write $\Gvp=\Gvp_1+\overline{\Gvp_2}$, where $\Gvp_1, \Gvp_2$ are in $K_{\Gt}$. We may also assume $\Gvp_1(0)=\Gvp_2(0)=0$.

It is easy to see that $$(A^{\Gt}_{\Gvp})^*=A^{\Gt}_{\bar{\Gvp}}.$$ Our approach to characterizing normal truncated Toeplitz operators starts with a computation of $$||A^{\Gt}_{\Gvp}u||^2-||(A^{\Gt}_{\Gvp})^*u||^2.$$
\begin{lem}\label{t1}
Let $\Gt$ be a non-constant inner function. Suppose $$\Gvp=\Gvp_1+\overline{\Gvp_2},$$ where $\Gvp_1, \Gvp_2$ are in $K_{\Gt}$. Then for every $u\in K_\Gt\cap H^{\infty}$,
\begin{align*}\label{t}
&||A^{\Gt}_{\Gvp}u||^2-||(A^{\Gt}_{\Gvp})^*u||^2\\
=&||P^{\perp}(\bar{\Gt}\Gvp_1 u)||^2-||P(\bar{\Gvp_1}u) ||^2-( ||P^{\perp}(\bar{\Gt}\Gvp_2 u)||^2-||P(\bar{\Gvp_2}u) ||^2).
\end{align*}
\end{lem}
\begin{proof}
By \eqref{Pt}, we have for every $u\in K_\Gt\cap H^\infty$
\begin{align*}
A^{\Gt}_{\Gvp}u &=P_{\Gt}(\Gvp u)\\
&=P(\Gvp u)-\Gt P(\bar{\Gt}\Gvp u)\\
&=\Gvp_1 u+P(\bar{\Gvp_2}u)-\Gt P(\bar{\Gt}\Gvp_1 u+\bar{\Gt}\bar{\Gvp_2} u )\\
&=\Gvp_1 u-\Gt P(\bar{\Gt}\Gvp_1 u) +P(\bar{\Gvp_2}u).
\end{align*}
Then
\begin{align*}
||A^{\Gt}_{\Gvp}u||^2 &=||(\Gvp_1 u-\Gt P(\bar{\Gt}\Gvp_1 u))+P(\bar{\Gvp_2}u) ||^2\\
&=||\Gvp_1 u-\Gt P(\bar{\Gt}\Gvp_1 u)||^2+||P(\bar{\Gvp_2}u) ||^2+2\, \m{Re} \langle \Gvp_1 u-\Gt P(\bar{\Gt}\Gvp_1 u), P(\bar{\Gvp_2}u) \rangle\\
&=||\bar{\Gt}\Gvp_1 u||^2-||P(\bar{\Gt}\Gvp_1 u)||^2+||P(\bar{\Gvp_2}u) ||^2+2\, \m{Re} \langle \Gvp_1 u-\Gt P(\bar{\Gt}\Gvp_1 u), P(\bar{\Gvp_2}u) \rangle\\
&=||P^{\perp}(\bar{\Gt}\Gvp_1 u)||^2+||P(\bar{\Gvp_2}u) ||^2+2\, \m{Re} \langle \Gvp_1 u-\Gt P(\bar{\Gt}\Gvp_1 u), P(\bar{\Gvp_2}u) \rangle.
\end{align*}
And
\begin{align*}
\langle \Gvp_1 u-\Gt P(\bar{\Gt}\Gvp_1 u), P(\bar{\Gvp_2}u) \rangle &=\langle \Gvp_1 u, P(\bar{\Gvp_2}u) \rangle-\langle \Gt P(\bar{\Gt}\Gvp_1 u), P(\bar{\Gvp_2}u) \rangle \\
&=\langle \Gvp_1 u, \bar{\Gvp_2}u \rangle-\langle P(\bar{\Gt}\Gvp_1 u), \bar{\Gt}P(\bar{\Gvp_2}u) \rangle\\
&=\langle \Gvp_1 u, \bar{\Gvp_2}u \rangle-\langle P(\bar{\Gt}\Gvp_1 u), \bar{\Gt}u\bar{\Gvp_2}-\bar{\Gt}P^{\perp}(\bar{\Gvp_2}u) \rangle\\
&=\langle \Gvp_1 u, \bar{\Gvp_2} u \rangle.
\end{align*}
Thus
\beq\label{l1}
||A^{\Gt}_{\Gvp}u||^2=||P^{\perp}(\bar{\Gt}\Gvp_1 u)||^2+||P(\bar{\Gvp_2}u) ||^2+2\, \m{Re} \langle \Gvp_1 u, \bar{\Gvp_2} u \rangle.
\eeq
Similarly
\beq\label{l2}
||(A^{\Gt}_{\Gvp})^*u||^2=||A^{\Gt}_{\Gvp_2+\bar{\Gvp_1}} u||^2=||P^{\perp}(\bar{\Gt}\Gvp_2 u)||^2+||P(\bar{\Gvp_1}u) ||^2+2\, \m{Re} \langle \Gvp_2 u, \bar{\Gvp_1}u\rangle.
\eeq
Subtracting \eqref{l2} from \eqref{l1}, we get the desired identity.
\end{proof}

For $w\in \DD$, let $$k_w(z)=\frac{1}{1-\bw z}$$ be the reproducing kernel of $H^2$.

First we show that if $A^{\Gt}_{\Gvp}$ is normal then $\Gvp_1/\Gvp_2$ is a unimodular function.
\begin{lem}\label{=1}
Let $\Gt$ be a non-constant inner function vanishing at $0$. Suppose $\Gvp=\Gvp_1+\overline{\Gvp_2}$, where $\Gvp_1, \Gvp_2$ are in $K_{\Gt}$, and $\Gvp_1(0)=\Gvp_2 (0)=0$. If $A^{\Gt}_{\Gvp}$ is normal then $$|\Gvp_1|=|\Gvp_2|,$$ a.e. on $\TT$.
\end{lem}
\begin{proof}
By Lemma \ref{t1}, $A^{\Gt}_{\Gvp}$ is normal implies  
\beq\label{=}
||P^{\perp}(\bar{\Gt}\Gvp_1 u)||^2-||P(\bar{\Gvp_1}u) ||^2=||P^{\perp}(\bar{\Gt}\Gvp_2 u)||^2-||P(\bar{\Gvp_2}u) ||^2,
\eeq
for every $u\in K_\Gt\cap H^{\infty}$. Take $u=1$, we get
$$
||P^{\perp}(\bar{\Gt}\Gvp_1)||^2-||P(\bar{\Gvp_1}) ||^2= ||P^{\perp}(\bar{\Gt}\Gvp_1)||^2-||P(\bar{\Gvp_2}) ||^2.
$$
Since \beq\label{1} P^{\perp}(\bar{\Gt}\Gvp_j)=\bar{\Gt}\Gvp_j,\eeq
and \beq\label{2}P(\bar{\Gvp_j})=0,\eeq
we have \beq\label{||}||\Gvp_1||=||\Gvp_2||.\eeq

Next we consider the reproducing kernels of $K_\Gt$:
$$k_w^{\Gt}(z)=\frac{1-\overline{\Gt(w)}\Gt(z)}{1-\bar{w}z},$$
and take $u=u_w =k_w^{\Gt}+1$ in \eqref{=}.
Using \eqref{1} and \eqref{2}, we have
\begin{align*}
||P^{\perp}(\bar{\Gt}\Gvp_j u_w)||^2=&||P^{\perp}(\bar{\Gt}\Gvp_j k_w^{\Gt})||^2+||P^{\perp}(\bar{\Gt}\Gvp_j)||^2+2\m{Re}\,\langle P^{\perp}(\bar{\Gt}\Gvp_j k_w^{\Gt}), P^{\perp}(\bar{\Gt}\Gvp_j ) \rangle\\
=&||P^{\perp}(\bar{\Gt}\Gvp_j k_w^{\Gt})||^2+||\bar{\Gt}\Gvp_j||^2+2\m{Re}\,\langle P^{\perp}(\bar{\Gt}\Gvp_j k_w^{\Gt}), \bar{\Gt}\Gvp_j  \rangle,
\end{align*}
and
$$
||P(\bar{\Gvp_j}u_w)||^2=||P(\bar{\Gvp_j}k_w^{\Gt})||^2+||P(\bar{\Gvp_j})||^2+2\m{Re}\,\langle P(\bar{\Gvp_j}k_w^{\Gt}), P(\bar{\Gvp_j})\rangle=||P(\bar{\Gvp_j}k_w^{\Gt})||^2.
$$
This together with Lemma \ref{t1} and \eqref{||} implies
\beq\label{3}
\m{Re}\,\langle P^{\perp}(\bar{\Gt}\Gvp_1 k_w^{\Gt}), \bar{\Gt}\Gvp_1\rangle =\m{Re}\,\langle P^{\perp}(\bar{\Gt}\Gvp_2 k_w^{\Gt}), \bar{\Gt}\Gvp_2 \rangle.
\eeq
Since $$k_w^{\Gt}=(1-\overline{\Gt(w)}\Gt)k_w,$$ we get
$$
P^{\perp}(\bar{\Gt}\Gvp_j k_w^{\Gt})= P^{\perp}(\bar{\Gt}\Gvp_j(1-\ol{\Gt(w)}\Gt) k_w)=P^{\perp}(\bar{\Gt}\Gvp_jk_w)-\ol{\Gt(w)}P^{\perp}(\Gvp_j k_w)=P^{\perp}(\bar{\Gt}\Gvp_jk_w).
$$
Hence
\begin{align*}
&\m{Re}\,\langle P^{\perp}(\bar{\Gt}\Gvp_j k_w^{\Gt}), \bar{\Gt}\Gvp_j\rangle\\
=&\m{Re}\,\langle P^{\perp}(\bar{\Gt}\Gvp_jk_w), \bar{\Gt}\Gvp_j\rangle\\
=&\m{Re}\,\langle \bar{\Gt}\Gvp_jk_w, \bar{\Gt}\Gvp_j\rangle\\
=&\m{Re}\,\int_{0}^{2\pi}\frac{ |\Gvp_j(e^{it})|^2}{1-\bw e^{it}}\frac{dt}{2\pi}\\
=&\int_{0}^{2\pi} |\Gvp_j(e^{it})|^2 \Big(\m{Re}\,\frac{1}{1-\bw e^{it}}\Big)\frac{dt}{2\pi}\\
=&\frac{1}{2}\int_{0}^{2\pi} |\Gvp_j(e^{it})|^2 \Big(1+\m{Re}\,\frac{1+\bw e^{it}}{1-\bw e^{it}}\Big)\frac{dt}{2\pi}\\
=&\frac{1}{2}||\Gvp_j||^2+\frac{1}{2}\int_{0}^{2\pi} |\Gvp_j(e^{it})|^2 \Big(\m{Re}\,\frac{1+\bw e^{it}}{1-\bw e^{it}}\Big)\frac{dt}{2\pi}\\
=&\frac{1}{2}(||\Gvp_j||^2+\widehat{|\Gvp_j|^2}(w)).
\end{align*}
The last equality holds because $$\m{Re}\,\frac{1+\bw e^{it}}{1-\bw e^{it}}$$ is the Poisson kernel at $w$. Here $\widehat{|\Gvp_j|^2}$ is the harmonic extension of the function $|\Gvp_j|^2$.
It follows from \eqref{3} and \eqref{||} that
$$
\widehat{|\Gvp_1|^2}(w)=\widehat{|\Gvp_2|^2}(w).
$$
Let $w\to\Gz\in\TT$ nontangentially, we see that
$$
|\Gvp_1|=|\Gvp_2|,
$$
a.e. on $\TT$.

\end{proof}

Let $U$ is the unitary operator on $L^2$ defined by $$Uh(z)=\bz \tih(z),$$  where $\tih(z)=h(\bz)$. 
Let $V_\Gt$ be the operator $$V_\Gt h=P(\Gt h),$$ for $h\in L^2$.
Consider the decomposition 
$$
[H^2]^{\perp}=\bar{\Gt}K_\Gt\oplus \bar{\Gt}[H^2]^{\perp}.
$$
It is easy to check that $V_\Gt$ maps $\bar{\Gt}K_\Gt$ onto $K_\Gt$, and maps $\bar{\Gt}[H^2]^{\perp}$ to $0$. 
Thus $V_\Gt$ maps $[H^2]^\perp$ onto $K_\Gt$.
Since $U$ maps $H^2$ onto $[H^2]^\perp$, we see that 
$$V_\Gt U: H^2\to K_\Gt$$ is also onto.

We shall use the following identity.
\begin{lem}\label{Han}
Let $\Gt$ be an inner function and let $g$ be in $H^2$. Then for every function $f\in H^\infty$
$$||P(\bg V_\Gt U f)||=||P^\perp(\bar{\Gt} g f^*)||,$$
where $f^*(z)=\ol{f(\bz)}$.
\end{lem}
\begin{proof}
Notice that for all $h\in L^2$, we have $$(Uh)^*=U(h^*)$$ and $$(Ph)^*=P(h^*).$$ Thus
\begin{align*}
P(\bg V_\Gt U f)&=P(\bg P(\Gt U f))=P(\bg \Gt U f)=P(\bz \Gt\bg \tif)\\
&=PU((\bar{\Gt}g)^*f)=P(U(\bar{\Gt}g f^*))^*\\
&=(PU(\bar{\Gt}g f^*))^*=(UP^\perp(\bar{\Gt}g f^*))^*.
\end{align*}
Here we used $PU=UP^\perp$ in the last equality.
Since $||h||=||h^*||,$ for all $h\in L^2$ and $U$ is an isometry, we get the desired identity.
\end{proof}

The following result is well-known (see e.g. \cite{tre15}*{Lemma 8}).
\begin{thm}\label{P}
If $f\in H^2$, then for every $w\in\DD$,
$$
P(\barf k_w) = \ol{f(w)} k_w.
$$
\end{thm}

Now we can prove the main result.
\begin{thm}\label{main}
Let $\Gt$ be a non-constant inner function vanishing at $0$. Suppose $\Gvp=\Gvp_1+\overline{\Gvp_2}$, where $\Gvp_1, \Gvp_2$ are in $K_{\Gt}$. Then $A^{\Gt}_{\Gvp}$ is normal if and only if either $$\Gvp_2-\Gvp_2(0)=\Ga(\Gvp_1-\Gvp_1(0))$$ or $$\Gvp_2-\Gvp_2(0)=\Ga\Gt(\ol{\Gvp_1}-\ol{\Gvp_1(0)}),$$
for some unimodular constant $\Ga$.
\end{thm}
\begin{proof}
We may assume $\Gvp_1(0)=\Gvp_2(0)=0$. Sufficiency follows easily from Lemma \ref{t1}.  

Suppose $A^{\Gt}_{\Gvp}$ is normal. 
By \eqref{=} and Lemma \ref{=1}, we have
\beq\label{u}
||P(\bar{\Gvp_1}u)||^2+||P(\bar{\Gt}\Gvp_1 u)||^2=||P(\bar{\Gvp_2}u)||^2+||P(\bar{\Gt}\Gvp_2 u)||^2,
\eeq
for every $u\in K_\Gt\cap H^\infty$.
According to the discussion before Lemma \ref{Han}, if we write $u=V_\Gt Uf$, where $f\in H^\infty$, \eqref{u} is equivalent to 
$$
||P(\bar{\Gvp_1}V_\Gt Uf)||^2+||P(\bar{\Gt}\Gvp_1 V_\Gt Uf)||^2=||P(\bar{\Gvp_2}V_\Gt Uf)||^2+||P(\bar{\Gt}\Gvp_2 V_\Gt Uf)||^2,
$$
for every $f\in H^\infty$.
Using Lemma \ref{Han} and that $f\mapsto f^*$ is a bijection on $H^\infty$, we have
\beq\label{u2}
||P^\perp(\bar{\Gt}\Gvp_1 f)||^2+||P^\perp(\bar{\Gvp_1}f)||^2=||P^\perp(\bar{\Gt}\Gvp_2 f)||^2+||P^\perp(\bar{\Gvp_2}f)||^2,
\eeq
for every $f\in H^\infty$.
By Lemma \ref{=1}, 
$$
||\bar{\Gt}\Gvp_1 f||=||\bar{\Gt}\Gvp_2 f||,
$$
and 
$$
||\bar{\Gvp_1}f||=||\bar{\Gvp_2}f||.
$$
We see that \eqref{u2} implies
\beq\label{H}
||P(\bar{\Gt}\Gvp_1 f)||^2+||P(\bar{\Gvp_1}f)||^2=||P(\bar{\Gt}\Gvp_2 f)||^2+||P(\bar{\Gvp_2}f)||^2,
\eeq
for every $f\in H^\infty$. 

Take $f=k_w$ in \eqref{H}. By Theorem \ref{P}, we get
\beq\label{m1}
|\Gvp_1(w)|^2+|(\Gt\bar{\Gvp_1})(w)|^2=|\Gvp_2(w)|^2+|(\Gt\bar{\Gvp_2})(w)|^2,
\eeq
for every $w\in \DD$. Here $(\Gt\bar{\Gvp_1})(w)$ means $\la \Gt\bar{\Gvp_1}, k_w\ra$.

On the other hand, using Lemma \ref{=1}, we have
\begin{align}\label{m2}
{\Gvp_1(w)}(\Gt\bar{\Gvp_1})(w)&=\la \Gvp_1(\Gt\bar{\Gvp_1}), k_w\ra=\la \Gt |\Gvp_1|^2, k_w\ra=\la \Gt |\Gvp_2|^2, k_w\ra\\
\nnb&=\la \Gvp_2(\Gt\bar{\Gvp_2}), k_w\ra={\Gvp_2(w)}(\Gt\bar{\Gvp_2})(w).
\end{align}
for every $w\in \DD$.

Multiplying both sides of \eqref{m1} by $|\Gvp_2(w)|^2$ and using \eqref{m2}, we have
\begin{align*}
|\Gvp_1(w)\Gvp_2(w)|^2+|\Gvp_2(w)(\Gt\bar{\Gvp_1})(w)|^2&=|\Gvp_2(w)|^4+|\Gvp_2(w)(\Gt\bar{\Gvp_2})(w)|^2\\
&=|\Gvp_2(w)|^4+|\Gvp_1(w)(\Gt\bar{\Gvp_1})(w)|^2,
\end{align*}
which is equivalent to 
$$
(|\Gvp_1(w)|^2-|\Gvp_2(w)|^2)(|(\Gt\bar{\Gvp_1})(w)|^2-|\Gvp_2(w)|^2)=0.
$$
Thus for every $w\in \DD$, either
$$
|\Gvp_1(w)|=|\Gvp_2(w)|,
$$
or 
$$
|\Gvp_2(w)|=|(\Gt\bar{\Gvp_1})(w)|.
$$
Then it follows from the properties of analytic functions that
either
$$
\Gvp_1=\Ga\Gvp_2,
$$
or 
$$
\Gvp_2=\Ga\Gt\bar{\Gvp_1},
$$
for some unimodular constant $\Ga$.

\end{proof}

\begin{rem}
The characterization given in Theorem \ref{main} is equivalent to that in Theorem \ref{ct}. In fact, if we write $\Gvp=\Gvp_1+\bar{\Gvp_2}+\Gvp(0)$, where $\Gvp_1, \Gvp_2$ are in $K_\Gt\cap zH^2$,  it is shown in \cite{ct14}*{Section 5} that $A^\Gt_\Gvp\in\mcB_\Gt^\Ga$ if and only if $\Gt\bar{\Gvp_2}=\Ga\Gvp_1$. 
\end{rem}

\bibliography{references}
\end{document}